\documentclass[10pt]{article}
\usepackage[square,authoryear]{natbib}
\usepackage{graphicx}
\usepackage{marsden_article}
\usepackage{eucal}
\usepackage{stmaryrd}
\usepackage{mathrsfs}
\usepackage{epstopdf,framed}
\usepackage{pdfsync}

\DeclareFontFamily{OT1}{pzc}{}
\DeclareFontShape{OT1}{pzc}{m}{it}{<-> s * [1.100] pzcmi7t}{}
\DeclareMathAlphabet{\mathpzc}{OT1}{pzc}{m}{it}

\begin{document}
\newtheorem{remark}[theorem]{Remark}

\title{Implicit Lagrange-Routh Equations and Dirac Reduction}


\newcommand{\todoHiro}[1]{\vspace{5 mm}\par \noindent
\framebox{\begin{minipage}[c]{0.95 \textwidth} \textcolor{red}{Hiro:} \tt #1
\end{minipage}}\vspace{5 mm}\par}


{
\author{Eduardo Garc\'{\i}a-Tora\~{n}o Andr\'{e}s\thanks{
email: egtoranoandres@gmail.com}
 \\Department of Mathematics, Faculty of Science\\
 The University of Ostrava\\
30. dubna 22, 701 03 Ostrava, Czech Republic\\[3mm]
 \and
Tom Mestdag\thanks{email: tom.mestdag@ugent.be}\\
Department of Mathematics\\ Ghent University\\ Krijgslaan 281, B--9000 Gent, Belgium\\[3mm]
\and
Hiroaki Yoshimura\thanks{email: yoshimura@waseda.jp}\\ Department of Applied Mechanics and Aerospace Engineering
\\ Waseda University
\\ Okubo, Shinjuku, Tokyo
\\ 169-8555, Japan
}
}

\maketitle
\vspace{-0.3in}

\begin{center}
\abstract{\vspace{0.2cm}
In this paper, we make a generalization of Routh's reduction method for Lagrangian systems with symmetry to the case where not any regularity condition is imposed on the Lagrangian. First, we show how implicit Lagrange-Routh equations can be obtained from the Hamilton-Pontryagin principle, by making use of an anholonomic frame, and how these equations can be reduced. To do this, we keep the momentum constraint implicit throughout and we make use of a Routhian function defined on a certain submanifold of the Pontryagin bundle. Then, we show how the reduced implicit Lagrange-Routh equations can be described in the context of dynamical systems associated to Dirac structures, in which we fully utilize a symmetry reduction procedure for implicit Hamiltonian systems with symmetry.
}
\end{center}

{\small
}
\tableofcontents

\section{Introduction}\label{sec:Intro}
There is no doubt that there exists a close relation between symmetries and conservation laws, which has been one of the fundamental motivations for many geometric approaches to mechanical systems. The symmetry group of a dynamical system can always be used to reduce the system to one with fewer variables. When symmetry, besides, leads to conserved quantities, it can be very advantageous to incorporate that property into the reduction process. For example, when the system is Hamiltonian on a symplectic manifold,  one first restricts the attention to the submanifold determined by the conserved momenta, and only later one takes the quotient of this submanifold by the remaining symmetry (which in general happens to be only a subgroup of the original symmetry group). This, in a few words, is the so-called {\it symplectic reduction theorem} (see, \cite{MaWe1974,Ma1992}).

Symplectic reduction may be applied to the standard case of classical Hamiltonian systems defined on the cotangent bundle. While this procedure has been thoroughly studied in the literature, its Lagrangian counterpart, the so-called Routh or tangent bundle reduction, has traditionally received much less attention, even though since its conception in~\cite{Ro1877, Ro1884} it has proven to be a valuable tool to obtain and discuss, e.g., the stability of steady motions or relative equilibria. A few modern approaches to the topic can be found in the papers~\cite{CrMe2008, LaCaVa2010, MaRaSc2000}. One of the drawbacks of these papers is that a regularity condition needs to be assumed. In this paper, we will focus upon Routh reduction within the context of Dirac structures, {\it without} assuming any regularity hypotheses.

For simplicity, let us consider for a moment the case of a Lagrangian $L(x,\dot x,\dot\theta)$ with a single cyclic coordinate $\theta$. The first step in Routh's procedure is to write the corresponding velocity $\dot \theta$ in terms of the remaining coordinates and velocities $(x,\dot x)$ by making use of the conservation law  $\partial L/\partial \dot\theta=\mu$, which follows from Noether's theorem. One then introduces the restriction $R^\mu(x,\dot x)$ of the function $L-\dot\theta(\partial L/\partial \dot\theta)$ to the level set where the momentum is $\mu$, the so-called {\it Routhian} (see, e.g., \cite{Ma1992}). With this function, one can observe that the remaining Euler-Lagrange equations of the coordinates $x$, again when constrained to the level set associated to $\mu$, are in fact Euler-Lagrange equations for the Routhian $R^\mu$. The end result of Routh's reduction method is therefore that it reduces the Euler-Lagrange equations of the Lagrangian $L(x,\dot x,\dot
 \theta)$
  to those of $R^\mu(x,\dot x)$ on a reduced configuration space. A crucial ingredient in the above process, however, is that the Lagrangian satisfies the  regularity condition  $(\partial^2 L/\partial^2 \dot\theta)\neq 0$, which is necessary for carrying out the first step. Routh's procedure and the regularity condition for it to be applicable can be generalized to arbitrary Lagrangians with a (possibly) non-Abelian symmetry group $G$. In this situation, the condition is often referred to as {\it $G$-regularity}.

It is easy to construct a Lagrangian which fails to be $G$-regular. The following example in $\mathbb{R}^2$ is taken from~\cite{LaCa2010}:
\[
L(x,y,v_x,v_y)=(v_x)^2+v_x v_y- V(x).
\]
Note that it has a cyclic coordinate $y$ (and therefore an Abelian symmetry group $G=\mathbb{R}$), but that it is not $G$-regular. Also linear  $G$-invariant Lagrangians will always fail to be $G$-regular. For example, the dynamics of $N$ vortices in the plane admit the following Lagrangian:
\[
L(z_l,{\dot z}_l)=\frac{1}{2i}\sum_k \gamma_k\big(\bar z_k\dot z_k-z_k\dot{\overline{z}}_k\big)-\frac{1}{2}\sum_n\sum_{k\neq n} \gamma_n\gamma_k\text{ln}|z_n-z_k|, \qquad z_l\in\mathbb{C},
\]
where $\gamma_k\in\mathbb{R}$ are parameters of the model; see \cite{Ch1978} for more details. This Lagrangian is clearly linear in its velocities and invariant under rotations of the vortices in the plane but not $G$-regular. Also in the context of plasma physics, linear Lagrangians often appear (see, e.g., \cite{Li1983}).
\medskip

The aim of this paper is to extend Routh's method to the most general case where not any regularity condition is imposed on the Lagrangian.  Our approach is based on the {\it Hamilton-Pontryagin principle} (as it is called in \cite{YoMa2006b}) which leads to an implicit formulation of the Euler-Lagrange equations on the so-called Pontryagin bundle $TQ\oplus T^*Q$. We will show that under the assumption of symmetry, we can reduce these implicit equations to a set of {\it reduced implicit Lagrange-Routh equations}. The key ingredient is that we can circumvent the hypotheses on regularity by keeping the momentum constraint implicit throughout. Our method involves a generalized Routhian function which is defined on a certain submanifold of the Pontryagin bundle rather than on a submanifold of the tangent bundle, as is commonly the case for $G$-regular systems. Implicit Lagrangian systems can be geometrically described in the framework of Dirac structures (see \cite{YoMa2006a}). The definition of Dirac structure in \cite{Co1990,Do1993} was originally inspired by the notion of Dirac brackets, which was coined by Paul Dirac (in the 1950s) for dealing
with constraints in the Hamiltonian setting when the given Lagrangian is singular (see e.g. \cite{CoWe1988,New1}). So from the very start, there has been a strong relation with singular Lagrangians and constraints. In the second part of the paper, we will show that  the reduced implicit Lagrange-Routh equations may be also formulated in terms of a Dirac structure, by considering a reduction method known for implicit Hamiltonian systems from \cite{Va1998,BlVa2001}.

For completeness, we mention that the paper by \cite{LaCa2010} also deals with the general case. However, these authors use a variational approach which is based on Hamilton's principle rather than on the Hamilton-Pontryagin principle. Therefore, it focusses on different aspects of the theory.
\medskip

This paper is organized as follows. In \S2, we review the derivation of the standard implicit Euler-Lagrange equations for a possibly degenerate (or singular) Lagrangian $L$ via the Hamilton-Pontryagin principle. We use a technique that is similar to the one that has been used in, for instance, \cite{CrMe2008, CrMe2010} to rewrite the implicit Euler-Lagrange equations in terms of an {\it anholonomic} frame. Once the implicit equations on a general frame are obtained, we specialize these expressions to a particular frame adapted to a given symmetry of the Lagrangian (\S\ref{sec:implicit}). For a prescribed value of  momenta, we find the implicit Lagrange-Routh equations, and express them in an invariant form. In \S\ref{sec:redsection} we reduce them to obtain the reduced implicit Lagrange-Routh equations. The regular cases are discussed in \S\ref{sec:cases}, where we illustrate our theory by showing how the reduced implicit Lagrange-Routh equations agree with those developed in the literature. \S\ref{sec:Dirac} rephrases the previous results in terms of reduction of Dirac structures. We show how the reduced implicit Lagrange-Routh equations correspond to a certain reduced implicit Hamiltonian system. Finally, in \S\ref{sec:examples}, some examples are shown.

\section{A version of the Hamilton-Pontryagin principle using anholonomic frames} \label{sec:anhol}
\paragraph{Hamilton-Pontryagin principles.}
Let $Q$ be a configuration manifold of a mechanical system with ${\rm dim}\,Q=n$. Coordinates on $Q$ are given by $q^\alpha$, fiber coordinates on $TQ$ and $T^*Q$ will be denoted by $v^\alpha$ and $p_\alpha$, respectively. In the following, the index $\alpha$ runs from $1$ to $n$ unless otherwise noted. The notations are chosen in such a way that we can make a notational difference between a general curve $(q(t),v(t))$ in $TQ$, and the lifted curve $(q(t),\dot q(t))$ in $TQ$ of a curve $q(t)$ in $Q$, where $t$ denotes the time in $I=\{t \in \mathbb{R} \mid a \le t \le b\}$.

The Hamilton-Pontryagin principle leads to an implicit form of the Euler-Lagrange equations of a possibly degenerate Lagrangian $L$. These equations follow from considering the following variational principle on $TQ\oplus T^*Q$:
\[
\delta\int^a_b \bigl[ L(q,v)+\langle p, ({\dot q}-v)\rangle \bigr]dt= \delta\int^a_b \bigl[ L(q^\alpha,v^\alpha)+p_\alpha({\dot q}^\alpha-v^\alpha) \bigr]dt=0,
\]
for variations of $\left(q(t),v(t),p(t)\right)$  where $q(t)$ has fixed endpoints and $v(t)$ and $p(t)$ are arbitrary. From this, we can easily conclude that a solution $(q(t),v(t),p(t))$ of the implicit Euler-Lagrange equations must satisfy
\begin{equation} \label{eq:IL1}
{\dot q}^\alpha= v^\alpha, \qquad  p_\alpha-\frac{\partial L}{\partial v^\alpha}=0, \qquad {\dot p}_\alpha= \frac{\partial L}{\partial q^\alpha}.
\end{equation}
See \cite{YoMa2006b} for more details.

\paragraph{Anholonomic frames and quasi-velocities.}
In this section, we shall rewrite the implicit Euler-Lagrange equations in terms of the so-called {\it quasi-velocities}. Lagrangian equations which involve quasi-velocities are often called {\em Hamel equations} in the literature (see, for instance, \cite{MarSch1993b, BlMaZe2009,CrMe2010}). We will need these expressions when we consider the Routhian in the following sections.

In the next paragraphs, we will need the natural lifts of vector fields on $Q$ to its tangent manifold and Pontryagin bundle, respectively. Let $(q^\alpha,v^\alpha)$ be the natural tangent bundle coordinates on $TQ$. If $X=X^\alpha (\partial /\partial q^\alpha)$ is a vector field on $Q$, then its complete lift $X^{\rm C}$ and vertical lift $X^{\rm V}$ are the vector fields on $TQ$, given by
\[
X^{\rm C} = X^\beta  \frac{\partial }{\partial q^\beta} + \frac{\partial X^\beta}{\partial q^\gamma} v^\gamma \frac{\partial}{\partial v^\beta}, \qquad X^{\rm V} = X^\beta \frac{\partial}{\partial v^\beta}.
\]
Likewise, its complete lift to $M=TQ\oplus T^*Q$ is the vector field
\[
X^{M} = X^\alpha \frac{\partial }{\partial q^\alpha} + \frac{\partial X^\beta}{\partial q^\alpha}v^\alpha \frac{\partial}{\partial v^\beta}   -\frac{\partial X^\beta}{\partial q^\alpha}p_\beta \frac{\partial}{\partial p_\alpha}.
\]
A standard reference for the properties of these vector fields is the book \cite{YaIs1973}. Given a vector field $Y=Y^{\alpha}(\partial/\partial q^{\alpha})$ on $Q$, we can form the linear function $\bar Y = Y^\alpha p_\alpha$ on $T^*Q \subset M$.  Likewise, for a 1-form $\theta =\theta_\alpha {d}q^\alpha$ we can define a linear function on $\bar\theta = \theta_\alpha v^\alpha$ on $TQ \subset M$. The following properties can then easily be verified:
\begin{equation} \label{action2}
X^{M} (\bar Y) =  \overline{[X,Y]}, \qquad X^{M} (\bar \theta) = \overline{{\mathcal L}_X\theta}.
\end{equation}

Quasi-velocities are fiber coordinates in $T_qQ$, defined with respect to a non-coordinate or {\bfi anholonomic frame}. Let $Z_\alpha = Z_\alpha^\beta (\partial/\partial {q^\beta})$ be a new basis for the set of vector fields on $Q$. This means that at each point $q$ the matrix $(Z^\alpha_\beta(q))$ has an inverse matrix, smoothly defined, which we will denote by  $(W^\alpha_\beta(q))$.
Each vector $v_q \in T_qQ$ can be expressed by $v_q = {\rm v}^\alpha Z_\alpha(q)$. The fiber coordinates $({\rm v}^\alpha)$ are then the {\bfi quasi-velocities} of $v_q$ with respect to the frame $\{Z_\alpha\}$. Their relation to the natural fiber coordinates is simply ${\rm v}^\alpha = W^\alpha_\beta v^\beta$.

The coordinate frame $\{\partial /\partial q^\alpha  \}$ is an example of a frame, whose corresponding quasi-velocities are simply the natural fiber coordinates $v^\alpha$. A measure for the deviation of a given frame $\{Z_\alpha\}$ from being a coordinate frame, is given by its {\bfi object of anholonomity} (see, e.g., \cite{Sc1954}), which is defined by the relation
\[
[Z_\beta,Z_\gamma] = R^\alpha_{\beta\gamma} Z_\alpha.
\]
The $R^\alpha_{\beta\gamma}$ are given in coordinates by the following expressions:
\begin{equation}
R^\alpha_{\beta\gamma}=\left(Z^\tau_\beta W^\alpha_\delta \frac{\partial Z^\delta_\gamma}{\partial q^\tau}-Z^\tau_\gamma W^\alpha_\delta \frac{\partial Z^\delta_\beta}{\partial q^\tau}\right)=
-\left(Z^\tau_\beta \frac{\partial W^\alpha_\delta}{\partial q^\tau}Z^\delta_\gamma -Z^\tau_\gamma \frac{\partial W^\alpha_\delta}{\partial q^\tau}Z^\delta_\beta\right). \label{eq:R_coord}
\end{equation}

We can lift the frame $\{Z_\alpha\}$ on $Q$ to the frame  $\{Z^{\rm C}_\alpha, Z^{\rm V}_\alpha\}$ on $TQ$.  In what follows, we will often make use of the following, easily verifiable, properties:
\begin{alignat}{3}\label{eq:action}
  {Z^{\rm C}_\alpha}(q^\beta) &=  Z^\beta_\alpha ,\qquad \qquad & {Z^{\rm V}_\alpha}(q^\beta) &= 0, \nonumber \\
  {Z^{\rm C}_\alpha}({\rm v}^\beta) &= - R^\beta_{\alpha\gamma}{\rm v}^\gamma, \qquad  \qquad & {Z^{\rm V}_\alpha}({\rm v}^\beta) &= \delta^\beta_\alpha.
\end{alignat}

The 1-forms  $W^\alpha= W^\alpha_\beta dq^\beta$ form a basis for 1-forms on $Q$.  If $(q^\alpha,p_\alpha)$ denote the natural coordinates on $T^*Q$, we  can also introduce {\bfi quasi-momenta} by means of ${\rm p}_\beta = Z^\alpha_\beta p_\alpha$, and we note that the natural pairing is preserved: $\langle p,v\rangle=p_\alpha v^\alpha ={\rm p}_\alpha {\rm v}^\alpha$.

The independent variations $\{ \delta q^\alpha, \delta v^\alpha, \delta p_\alpha  \}$ form a basis for all variations on $M=TQ \oplus T^*Q$. We can change this to a new basis, adjusted to the new frame $\{Z_\alpha, W^\alpha \}$ on $M$. If we denote
\[
w^\alpha =  W^\alpha_\beta \delta q^\beta,
 \]
then $\{ w^\alpha, \delta {\rm v}^\alpha, \delta {\rm p}_\alpha \}$ will be all independent variations and hence this set will form a new basis for all variations on $M$. Here one should think of the quasi-velocities ${\rm v}^\alpha$ as functions on $TQ$ (or on $M$), and therefore
\[
\delta {\rm v}^\alpha = \frac{\partial W^\alpha_\beta}{\partial q^\gamma} v^\beta \delta q^\gamma + W^\alpha_\beta \delta v^\beta.
\]

\paragraph{The implicit Lagrangian systems with quasi-velocities and quasi-momenta.}
The direct computation using~\eqref{eq:R_coord} yields the variation of a Lagrangian $L$ on $TQ$ as
\begin{equation} \label{eq:varL}
\delta L =  \frac{\partial L}{\partial q^\alpha}\delta q^\alpha + \frac{\partial L}{\partial {v}^\alpha}\delta {v}^\alpha = {Z^{\rm C}_\alpha}(L) w^\alpha + {Z^{\rm V}_\alpha}(L) (\delta {\rm v}^\alpha + R^\alpha_{\beta\gamma} {\rm v}^\gamma w^\beta).
\end{equation}

We will now give a version of the Hamilton-Pontryagin principle that makes use of the anholonomic frame. We are looking for a curve $(q^{\alpha}(t), {\rm v}^\alpha(t), {\rm p}_{\alpha}(t))$ in $M=TQ\oplus T^*Q$, namely the one whose base curve is $q(t)$ and whose fiber coordinates are given by the curve $({\rm v}^\alpha(t),{\rm p}_\alpha(t))$ in quasi-velocities and quasi-momenta  which satisfies the following variational principle
\[
0=\delta\int^a_b \bigl[ L(q,v)+\langle p, {\dot q}-v\rangle \bigr]dt= \delta\int^a_b \bigl[ L(q,{\rm v})+{\rm p}_\alpha({\rm u}^\alpha - {\rm v}^\alpha) \bigr]dt.
\]
Here ${\rm u}^\alpha(t)$ stand for the quasi-velocities of the lifted curve $\dot q(t)$ in $TQ$, namely, ${\rm u}^\alpha(t) = W^\alpha_\beta(q(t)) {\dot q}^\beta(t)$.

If we take the above expression~\eqref{eq:varL} for $\delta L$ into account, we obtain
\[
0 = \int^a_b \left[ {Z^{\rm C}_\alpha}(L) w^\alpha + {Z^{\rm V}_\alpha}(L) (\delta {\rm v}^\alpha + R^\alpha_{\beta\gamma} {\rm v}^\gamma w^\beta) + \delta {\rm p}_\alpha({\rm u}^\alpha - {\rm v}^\alpha) + {\rm p}_\alpha \delta {\rm u}^\alpha - {\rm p}_\alpha \delta {\rm v}^\alpha     \right]dt.
\]
First, we compute
\begin{eqnarray*}
\int^a_b   {\rm p}_\alpha \delta {\rm u}^\alpha     dt &=& \int^a_b  \left[  {\rm p}_\alpha    \frac{\partial W^\alpha_\beta}{\partial q^\gamma} {\dot q}^\beta \delta q^\gamma + {\rm p}_\alpha W^\alpha_\beta \delta {\dot q}^\beta \right] dt \\[2mm]
&=& \int^a_b  \left[  {\rm p}_\alpha  \left(  \frac{\partial W^\alpha_\beta}{\partial q^\gamma} - \frac{\partial W^\alpha_\gamma}{\partial q^\beta} \right) {\dot q}^\beta \delta q^\gamma  - {\dot {\rm p}}_\alpha  w^\alpha  \right] dt \\[2mm]
&=&  \int^a_b  \left[  ({\rm p}_\alpha R^\alpha_{\mu\tau} {\rm u}^\mu  - {\dot {\rm p}}_\tau)  w^\tau  \right] dt.
\end{eqnarray*}
In the above, we have used integration by parts and the fact that $\delta q^\alpha(a) = \delta q^\alpha(b) =0$. We have also made use of the expression \eqref{eq:R_coord} for $R^\alpha_{\mu\tau}$. Then, by substituting this into our variational principle, it follows that
\begin{equation*}
\begin{split}
0 &= \int^a_b \left[ \Big( {Z^{\rm C}_\tau}(L) + {\rm p}_\alpha R^\alpha_{\mu\tau} {\rm u}^\mu +  {Z^{\rm V}_\alpha}(L) R^\alpha_{\tau\gamma} {\rm v}^\gamma  - {\dot {\rm p}}_\tau   \Big)  w^\tau \right. \\
& \hspace{3cm}\left.   + \Big({Z^{\rm V}_\alpha}(L) - {\rm p}_\alpha\Big)\delta {\rm v}^\alpha  - \Big({\rm v}^\alpha - {\rm u}^\alpha\Big) \delta {\rm p}_\alpha   \right]dt
\end{split}
\end{equation*}
holds for all variations $w^\tau$, $\delta {\rm v}^\alpha$ and $\delta{\rm p}_\alpha$. Since these are all independent, we conclude that the curve $(q^\alpha(t),{\rm v}^\alpha(t), {\rm p}_\alpha(t))$ should satisfy
\begin{equation}\label{eq:IL}
 {\rm v}^\alpha= {\rm u}^\alpha, \qquad  {\rm p}_\alpha={Z^{\rm V}_\alpha}(L), \qquad {\dot {\rm p}}_\alpha= {Z^{\rm C}_\alpha}(L).
\end{equation}
These are the {\bfi  implicit Euler-Lagrange equations in quasi-velocities and quasi-momenta}. Remark that the last equation takes such a simple form, only because we have already made use of the first two equations.

\section{The implicit Lagrange-Routh equations} \label{sec:implicit}
\paragraph{The Lie group action on $Q$.}
Assume that $G\times Q \to Q$ is a free and proper action of a possibly non-Abelian Lie group $G$ on $Q$, so that we can regard $Q \to Q/G$ as a principal $G$-bundle. We will assume in this section that the Lagrangian $L:TQ \to \mathbb{R}$ is invariant under the symmetry group $G$. We use $x^i$ for coordinates on the {\it shape space} $Q/G$ and $(q^\alpha)=(x^i,\theta^a)$ for coordinates on $Q$. We denote the corresponding coordinates on the tangent bundle by $(v^i,v^a)$. As before, we use the notation $(x(t),{\dot x}(t))$ to denote the lifted curve on $T(Q/G)$ of a curve $x(t)$ in $Q/G$.

We now choose a principal connection on $Q\to Q/G$ and let $X_i$ denote the horizontal lift of the coordinate vector fields $\partial/\partial x^i$ on the shape space $Q/G$ with respect to this connection. In terms of the coordinates introduced above, we have
\[
X_i = \frac{\partial }{\partial x^i} - \Lambda_i^a \frac{\partial}{\partial \theta^a},
\]
where $\Lambda_i^a$ are functions on $Q$ (the connection coefficients) and we employ the symbol $\Lambda$ to denote the {\it connection one-form} on $TQ$ which takes values in $\mathfrak{g}$.
\medskip

Let $\{E_a\}$ be a basis of the Lie algebra $\mathfrak{g}$, and $\{E^a\}$ the corresponding dual basis of $\mathfrak{g}^*$. We denote by $C^c_{ab}$ the structure constants of $\mathfrak{g}$, $[E_a,E_b]=C^c_{ab}E_c$.
 The fundamental vector fields associated to the action will be denoted by $\{{\tilde E}_a\}$. We can express them as
\[
{\tilde E}_a = K^b_a \frac{\partial }{\partial \theta^b},
\]
for some functions $K^b_a$ on $Q$, often called the coefficients of the infinitesimal generator map.

We  suppose throughout  the paper that $G$ is connected. This has the advantage that invariance of functions and tensor fields can be checked by the vanishing of the Lie derivatives of these functions and tensor fields, in the direction of the fundamental vector fields of the action.

The action on $Q$ lifts to actions on $TQ$, $T^*Q$ and $M=TQ\oplus T^*Q$. In each case, it is well-known that the fundamental vector fields of the lifted actions are given by the complete lifts (to $TQ$, $T^*Q$ and $M$, respectively) of the fundamental vector fields on $Q$. In the case of the action on $M$ the fundamental vector fields are therefore (linear combinations of) the vector fields $\tilde{E}_a^{M}$. Since $G$ is supposed to be connected, a function $F$ on $M$ is invariant if, and only if, ${{\tilde E}_a}^{M} (F)=0$.

\paragraph{The implicit Lagrange-Routh equations.}
Let us rewrite the implicit Lagrange equations~\eqref{eq:IL} by making use of a specific anholonomic frame. First we consider the frame $\{Z_\alpha\} = \{ X_i, {\tilde E}_a \}$ on $Q$. This corresponds to the so-called {\bfi moving frame} in literature. We denote the corresponding quasi-velocities and quasi-momenta by $({\rm v}^i,{\tilde{\rm v}}^a)$ and $({\rm p}_i,{\tilde {\rm p}}_a)$. In fact, since the vector fields  $X_i$ are assumed to project onto the coordinate fields on $Q/G$, the quasi-velocities ${\rm v}^i$ can be naturally identified with the natural fiber coordinates $v^i$ on $T(Q/G)$. The curve ${\rm u}^i(t)$ that appears in the equations~\eqref{eq:IL} is then simply the lifted curve ${\dot x}^i(t)$ of the curve $x^i(t)$ in $Q/G$. Similarly, the quasi-momenta ${\rm p}_i$ can be identified with the momenta $p_i$ of $T^*(Q/G)$, but this identification is not canonical since it depends on the choice of the connection $\Lambda$. From now on, for simplicity, we will use the notation $v^i$ and $p_i$ to denote the corresponding quasimomenta.

The brackets of the frame are given by:
\begin{equation*}
[X_i,X_j] = B_{ij}^a {\tilde E}_a, \qquad  [X_i, {\tilde E}_a]=0, \qquad [{\tilde E}_a, {\tilde E}_b] = -C^c_{ab} {\tilde E}_c,
\end{equation*}
where the $B_{ij}^a$ stand, up to a sign, for the curvature coefficients of the principal connection $\Lambda$ (this is the convention in \cite{CrMe2008}, but differs from e.g.~\cite{MaRaSc2000}). The relation $[X_i, {\tilde E}_a]$ is a consequence of the invariance of the $X_i$.

The Lagrangian is invariant if, and only if,  ${{\tilde E}_a}^{M} (L)= {{\tilde E}_a}^{\rm C} (L)=0$. The implicit Euler-Lagrange equations (\ref{eq:IL}) are therefore
 \begin{equation*}
 \begin{split}
 \tilde {\rm v}^a& =\tilde {\rm u}^a, \qquad  \tilde{\rm p}_a={\tilde E}_a^{\rm V}(L), \qquad  {\dot{\tilde{\rm p}}}_a= 0,\\
  v^i & = {\dot x}^i, \qquad  p_i=X_i^{\rm V}(L), \qquad  {\dot p}_i= X_i^{\rm C}(L).
 \end{split}
 \end{equation*}

From the top row, we see that $\tilde{\rm p}_a$ is constant along solutions, say $\tilde{\rm p}_a =\mu_a$, with $\mu=\mu_a E^a \in \mathfrak{g}^*$. We can define a {\bfi generalized Routhian}, as the function on $TQ$ given by
\begin{equation}\label{Routhian}
R^\mu (q,v)= L(q,v) - \mu_a {\tilde {\rm v}}^a.
\end{equation}
It is, however, not the standard definition of the Routhian function as one may find in, for instance, \cite{MaRaSc2000,CrMe2008,LaCaVa2010}. We will clarify the relation between these two definitions later in in \S\ref{sec:cases}. The advantage of the current definition is that it allows us to keep the momentum constraint implicit.

From the relations~\eqref{eq:action} of the previous section, we obtain
 \begin{equation}\label{eq:Rder}
 \begin{split}
  X_i^{\rm C}(R^\mu)&= X_i^{\rm C}(L) + \mu_a B^a_{ij}v^j , \qquad\quad  X_i^{\rm V}(R^\mu) = X_i^{\rm V}(L),  \\
  {\tilde E}_a^{\rm C}(R^\mu)&=   - \mu_c C^c_{ab}{\tilde {\rm v}}^b, \qquad\qquad\qquad  {\tilde E}_a^{\rm V}(R^\mu) = {\tilde E}_a^{\rm V}(L) - \mu_a.
 \end{split}
 \end{equation}

Therefore, a solution of the implicit Euler-Lagrange equations~\eqref{eq:IL} is a curve
\[
(x^i(t),\theta^a(t), v^i(t), \tilde{\rm v}^a(t), p_i(t),\tilde{\rm p}_a(t)):I\subset \mathbb{R}\to M=TQ\oplus T^*Q
\]
satisfying
 \begin{equation}\label{eq:LR}
 \begin{split}
  \tilde{\rm v}^a&=\tilde{\rm u}^a, \qquad   \tilde{E}^{\rm V}_a(R^\mu)  = 0, \qquad \;\;  \tilde{\rm p}_a= \mu_a, \\
  {v}^i&={\dot x}^i, \qquad     X_i^{\rm V}(R^\mu) =p_i, \qquad   {\dot p}_i = X_i^{\rm C}(R^\mu) - \mu_a B^a_{ij}v^j.
 \end{split}
 \end{equation}

We will call these equations the {\bfi implicit Lagrange-Routh equations}. The terminology Lagrange-Routh equations is adopted from \cite{MaRaSc2000}.

\paragraph{The implicit Lagrange-Routh equations in invariant form.}
We will restrict our attention to one specific level set of momentum. Consider the canonical momentum map $J: T^*Q \to \mathfrak{g}^*$ of the $G$-action on $Q$, and fix a value $\mu\in\mathfrak{g}^*$. Let $M_\mu$ denote the submanifold $TQ \oplus J^{-1}(\mu)$ in $M=TQ\oplus T^*Q$. Local coordinates on $M_\mu$ are then $(q,v^i,\tilde{\rm v}^a,p_i)$, the $\tilde{\rm p}_a$ being fixed by the value $\mu\in\mathfrak{g}^*$.

The $G$-action on $M$ restricts to a $G_\mu$-action on $M_\mu$, where $G_\mu$ stands for the isotropy group. We will describe how solutions of the implicit Lagrange-Routh equations~\eqref{eq:LR} which happen to lie on $M_\mu$ can be projected to  curves in $M_\mu/G_\mu$,  satisfying some reduced equations. In order to do that, we need to rewrite them in such a way that all involved terms are given by $G_\mu$-invariant functions. When that is the case, these $G_\mu$-invariant equations will project  to equations on $M_\mu/G_\mu$.

Let $\xi=\xi^a E_a \in \mathfrak{g}$ be an arbitrary element of $\mathfrak{g}$, then it follows from~\eqref{eq:Rder} that ${\tilde \xi}^{\rm C}(R^\mu)=   - \xi^a \mu_c C^c_{ab}\tilde{\rm v}^b$. From this we see that $\tilde{\xi}^{\rm C}(R^\mu)=0$, if and only if, $\xi\in\mathfrak{g}_\mu$. Therefore, we see that $R^\mu$ is (only) $G_\mu$-invariant. The Routhian $R^\mu$ can thus be identified with a reduced function on $TQ/G_\mu$ (for which we shall use the same notation).

We next define local coordinates on $M_\mu/G_\mu$. It is easy to see that the quasi-velocities $v^i$ and the quasi-momenta $p_i$ are $G_\mu$-invariant functions on $M_\mu$. Indeed, from~\eqref{action2} and~\eqref{eq:action}, we have
\[
\tilde{E}^{M}_a(v^i) = 0, \qquad \tilde{E}^{M}_a(p_i) = \overline{[{\tilde E}_a,X_i]} =0.
\]
The last property is based on the observation that $\overline{X_i} = p_i$. The above expressions show that $v^i$ and $p_i$ (thought of as coordinate functions on $M$) are $G$-invariant, and thus also $G_\mu$-invariant functions on $M$ (and therefore also on $M_\mu$). On the other hand, the  quasi-velocities $\tilde{\rm v}^a$ are not $G_\mu$-invariant functions on $M_\mu$ since
\[
\tilde{E}^{M}_a(\tilde{\rm v}^b) = - C^b_{ac} \tilde{\rm v}^c.
\]

To overcome this issue we introduce a new  frame that is completely $G$-invariant (and therefore also $G_\mu$-invariant), see also \cite{CrMe2008}. This coincides in the literature with the so-called {\bfi body-fixed frame}. Consider a new set of vector fields, given by  ${\hat E}_a = A^b_a{\tilde E}_b$. The following reasoning shows that there exists a matrix $(A^b_a)$ of functions on $Q$ for which these vector fields are all invariant. They may be invariant if and only if $0=[{\tilde E}_a,
{\hat E}_b] = \Big({\tilde E}_a(A^c_b) - C^c_{ad}A^d_b\Big) {\tilde
E}_c$. The {\it integrability condition} that is needed for the PDE equation
\[
{\tilde E}_a(A^c_b) - C^c_{ad}A^d_b = 0
\]
to have a solution $A^a_b$ is satisfied by virtue of the Jacobi identity of the Lie bracket on $\mathfrak{g}$. We can therefore claim that, at least locally, the above PDE has a solution  for which $A=(A_a^b)$ is non-singular,
and for which $A$ is the identity on some specified local section of
$\pi: Q \to Q/G$.

An explicit way to define the vector fields ${\hat E}_a$, and the one we will use henceforth, is as follows. Let $U\subset Q/G$ be an open set over which $Q$ is locally trivial. Then the fibration is $\pi: U\times G \to U$, and the action is given by $\psi_g (x,h) = (x,gh)$. We can define
${\hat
E}_a: (x,g) \mapsto \widetilde{\mathrm{Ad}_{g} E}_a(x,g) =
T\psi_g \big({\tilde E}_a (x,e)\big)$, where $\mathrm{Ad}_g:\mathfrak{g}\to \mathfrak{g}$ is the adjoint action. In coordinates, we write
\[
{\hat E}_a = L^b_a \frac{\partial }{\partial \theta^b},
\]
where $L^b_a$ are functions on $Q$.

Another way to think of these two frames is the following: If $Q$ is the Lie group $G$, and the action is given by left multiplication then the tilde-vector fields coincide with a basis of right-invariant vector fields, while the hat-vector fields are all left-invariant.

The quasi-velocities $(v^i,\hat{\rm v}^a)$ with respect to $\{X_i, \hat{E}_a\}$ are all invariant functions on $M$, since now
\[
\tilde{E}^{M}_a(\hat{\rm v}^b) =  0.
 \]
We can therefore take $([q]_{G_\mu}, v^i,\hat{\rm v}^a, p_i)$ for our local coordinates on $M_\mu/ G_\mu$, where $[q]_{G_\mu}$ stands for the coordinates of the orbit of $q$ under the $G_\mu$-action. In a more global interpretation, we have split the quotient $M_\mu/ G_\mu$ by making use of the principal connection $\Lambda$ as
\[
M_\mu/G_\mu \simeq \left(Q/G_\mu \right)\times_{Q/G}\bigl(T(Q/G)\oplus \tilde{\mathfrak{g}} \oplus T^*(Q/G)\bigr),
\]
where $\tilde{\mathfrak{g}}$ is the adjoint bundle. In fact, the coordinates above correspond to fiber coordinates with respect to this identification (see~\cite{LaCa2010} for more details).

We now check whether all the terms that appear in the implicit Lagrange-Routh equations~\eqref{eq:LR} are $G_\mu$-invariant. For example, the function ${\tilde E}^{\rm V}_a(R^\mu) $ is not $G_\mu$-invariant, but the function $\hat {E}_b^{\rm V}(R^\mu) = A_b^a \tilde{E}^{\rm V}_a(R^\mu) $ is. Indeed, for $\xi\in\mathfrak{g}_\mu$,
\[
\xi^c\tilde{E}^{\rm C}_c \left(  \hat{E}^{\rm V}_b(R^\mu) \right) =   \hat{E}^{\rm V}_b \left( \xi^c \tilde{E}^{\rm C}_c  (R^\mu) \right) =0,
\]
where we have used that $ [ \hat{E}^{\rm V}_b, \tilde{E}^{\rm C}_c] = [ \hat{E}_b, \tilde{E}_c]^{\rm C} =0$, and where we have also used that the Routhian is $G_\mu$-invariant. We conclude that we need to replace the equation $\tilde{E}_a^{\rm V}(R^\mu) =0$ in~\eqref{eq:LR} by the equivalent equation $\hat{E}_b^{\rm V}(R^\mu) = A_b^a \tilde{E}_a^{\rm V}(R^\mu) =0$ to obtain a  $G_\mu$-invariant equation. With a similar argument we can show that the functions $X_i^{\rm C}(R^\mu)$ and $X_i^{\rm V}(R^\mu)$ are all $G_\mu$-invariant.

Remark that, since $[X_i,X_j]$ is an invariant vector field, we must have that $[\tilde{E}_c^{\rm C} , B^a_{ij}\tilde{E}_a^{\rm C}] =0 $, or equivalently
\[
\tilde{E}_c^{\rm C} ( B^b_{ij})  - B^a_{ij} C^b_{ca} =0.
\]
Using this, we can now check that the function
$\mu_a B^a_{ij}v^j$ is $G_\mu$-invariant; namely, for $\xi\in\mathfrak{g}_\mu$, we find
\[
\xi^c \tilde{E}_c^{\rm C} (\mu_a B^a_{ij}v^j) = \xi^c \mu_a  C^a_{cd} B^d_{ij}v^j=0.
\]

To conclude, apart from the defining relation $\tilde{\rm p}_a= \mu_a$ for $M_\mu$,  the following system of equations is equivalent to the implicit Lagrange-Routh equations~\eqref{eq:LR} and consists only of  $G_\mu$-invariant equations:
 \begin{equation}\label{eq:invLR}
 \begin{split}
  \hat{\rm v}^a & =\hat{\rm u}^a, \qquad  \hat{E}^{\rm V}_a(R^\mu) = 0,  \\
  v^i& =\dot{x}^i, \qquad  {X}^{\rm V}_i(R^\mu)=p_i, \qquad  \dot{p}_i = {X}^{\rm C}_i(R^\mu) - \mu_a B^a_{ij}v^j.
 \end{split}
 \end{equation}

All these equations correspond to equations on $M_\mu/G_\mu$, whose coordinate expressions will be obtained in the next section.

We finish this paragraph with a coordinate expression for the equations $\hat{\rm v}^a=\hat{\rm u}^a$. Recall first that $\hat{\rm u}^a$ is the quasi-velocity corresponding to the lifted curve $(q,\dot q)$. If we write
\[
\dot x^i\frac{\partial }{\partial x^i} + \dot \theta^a\frac{\partial}{\partial \theta^a} = \dot x^i X_i + (L^{-1})^a_b(\dot\theta^b+\Lambda^b_i\dot x^i)\hat{E}_a,
\]
we can conclude that the equation $\hat{\rm v}^a=\hat{\rm u}^a$ is equivalent with the equation $\hat{\rm v}^b L^a_b = \dot{\theta}^a+  \dot{x}^i \Lambda^a_i$.

\section{Reduction of the implicit Lagrange-Routh equations}\label{sec:redsection}
We shall use the residual $G_\mu$-symmetry of the equations~\eqref{eq:invLR} to drop them to $M_\mu/G_\mu$. To do this, we will make use of an invariant decomposition (with respect to the adjoint action of $G_\mu$) of the Lie algebra $\mathfrak{g}=\mathfrak{g}_\mu\oplus (\mathfrak{g}/\mathfrak{g}_\mu)$ meaning that, for all $g\in G_\mu$, we have $\textcolor{red}{\mathrm{Ad}}_{g}(\mathfrak{g}/\mathfrak{g}_\mu)\subset (\mathfrak{g}/\mathfrak{g}_\mu)$. In terms of coordinates, this splitting will be denoted as follows: we choose a basis $\{E_a\}=\{E_A,E_I\}$ of $\mathfrak{g}$ in such a way that $\{E_A\}$ represents a basis of $\mathfrak{g}_\mu$ and $\{E_I\}$ a basis of $(\mathfrak{g}/\mathfrak{g}_\mu)$.

An invariant splitting $\mathfrak{g}=\mathfrak{g}_\mu\oplus (\mathfrak{g}/\mathfrak{g}_\mu)$ of the Lie algebra as above, together with a principal connection on the bundle $Q\to Q/G$, induces a principal connection on the bundle $Q\to Q/G_\mu$. Indeed, let $\pi_{\mathfrak{g}_\mu}$ denote the projector $\pi_{\mathfrak{g}_\mu}:\mathfrak{g}\to \mathfrak{g}_\mu$ with respect to the previous decomposition. Define the principal connection 1-form $\Lambda^\mu:=\pi_{\mathfrak{g}_\mu}\circ \Lambda:TQ\to \mathfrak{g}_\mu$, where $\Lambda:Q\to\mathfrak{g}$ is the principal connection 1-form on the bundle $Q\to Q/G$. This is a well-defined connection 1-form because
\begin{itemize}
\item[(1)] If $\xi\in\mathfrak{g}_\mu$, then certainly $\Lambda^\mu(\tilde\xi)=\pi_{\mathfrak{g}_\mu}(\xi)=\xi$;
\item[(2)] If $g\in G_\mu$, then since $\mathrm{Ad}_g$ is linear and $\mathrm{Ad}_g(\mathfrak{g}_\mu)=\mathfrak{g}_\mu$, it follows that $\pi_{\mathfrak{g}_\mu}\circ \mathrm{Ad}_g=\mathrm{Ad}_g\circ\pi_{\mathfrak{g}_\mu}$.
\end{itemize}

Therefore, $\Lambda^\mu$ is $G_\mu$-equivariant. Although we will not explicitly make use of it, we mention that an (Ehresmann) connection on $Q/G_\mu\to Q/G$ can be directly obtained by projecting the connection $\Lambda$ (note that it is well-defined due to the equivariance of $\Lambda$). This connection, together with $\Lambda$ and $\Lambda_\mu$, plays a role when looking at variational principles in the context of Routh reduction (see also~\cite{LaCa2010} and~\cite{MaRaSc2000}).

Using that $\mathfrak{g}_\mu$ is a Lie subalgebra, we have $C^J_{AB}=0$. On the other hand, from the definition of $\mathfrak{g}_\mu$, we also get $C^c_{Ab}\mu_c=0$. Finally the fact that the term $(\mathfrak{g}/\mathfrak{g}_\mu)$ is invariant leads to the relation $C^C_{ab}=0$. We will need these relations later in the paper.

We will use coordinates $(\theta^a) = (\theta^A, \theta^I)$ such that the fibers of $G \to G/G_\mu$ are given by $\theta^I= \mbox{constant}$. Then, there are functions $K^a_b$ on $Q$ such that
\[
{\tilde E}_A = K^B_A \frac{\partial }{\partial \theta^B}, \qquad {\tilde E}_I =
K^B_I \frac{\partial}{\partial \theta^B} + K^J_I \frac{\partial}{\partial \theta^J}.
\]
We remark that, by construction, $K^J_A=0$. Likewise, we can set \[
{\hat E}_A = L^B_A \frac{\partial }{\partial \theta^B}, \qquad {\hat E}_I =
L^B_I \frac{\partial}{\partial \theta^B} + L^J_I \frac{\partial}{\partial \theta^J}.
\]

Since $X^{\rm V}_j$ is a $G_\mu$-invariant vector field on $TQ$, it can be thought of as a vector field on $TQ/G_\mu$. This reduced vector field will be completely determined by its action on the coordinate functions $(x^i, \theta^I, v^i, \hat{\rm v}^a)$, which define the coordinates on $TQ/G_\mu$. We have, again using~\eqref{action2} and~\eqref{eq:action},
\[
X^{\rm V}_j(x^i)=0, \quad X_j^{\rm V}(\theta^I)=0, \quad X_j^{\rm V}(v^i) = \delta_j^i,\quad X_j^{\rm V}(\hat{\rm v}^a)=0.
\]
The reduced vector field of $X_j^{\rm V}$ is therefore the vector field $\partial/\partial v^j$ on $M_\mu/G_\mu$. Likewise, the reduced vector fields of $ \hat{E}_a^{\rm V}$ are $\partial/\partial {\hat{\rm v}^a}$. It also follows that
\[
X_j^{\rm C}(x^i)=\delta_j^i, \quad X_j^{\rm C}(\theta^I)=-\Lambda_j^I, \quad X_j^{\rm C}(v^i) = 0  ,\quad X_j^{\rm C}(\hat{\rm v}^a)= -\hat B^a_{jk}v^k ,
\]
where $[X_i,X_j]= \hat B^a_{ij}{\hat E}_a$ (by construction, we have $\hat B^a_{ij}L^b_a=B^a_{ij}K^b_a$, or equivalently $ B^a_{ij} = A^a_b \hat B^b_{ij}$). Therefore the reduction of $X_j^{\rm C}$ to a vector field on $TQ/G_\mu$ is
\[
 \frac{\partial}{\partial x^j} - \Lambda_j^I \frac{\partial}{\partial \theta^I}  - \hat B^a_{jk}v^k \frac{\partial}{\partial \hat{\rm v}^a}.
\]
Summarizing, we have proved the following:
\begin{proposition}
 A curve $(x^i(t),\theta^I(t),v^i(t),\hat{\rm v}^a(t), p_i(t))$ in $M_\mu/G_\mu$ is a solution of the {\bfi reduced implicit Lagrange-Routh equations} if it satisfies:
\begin{equation} \label{LRred}
\begin{split}
&{\dot x}^i=v^i, \quad\;\; {\dot \theta}^I=\hat{\rm v}^J L^I_J - {\dot x}^i \Lambda^I_i, \quad\;\;
{\dot p}_i=\frac{\partial R^\mu}{\partial x^i} - \Lambda^I_i \frac{\partial R^\mu}{\partial \theta^I} - \mu_a B_{ij}^a {\dot x}^j, \\
& p_i = \frac{\partial R^\mu}{\partial v^i}, \quad\quad \frac{\partial R^\mu}{\partial \hat{\rm v }^a} =0.
\end{split}
\end{equation}
\end{proposition}

None of the above equations depends  explicitly on $\theta^A$. Therefore, the above equations determine the reduced curve on $M_\mu/G_\mu$.

Once we have solved these reduced equations for $(x^i(t),\theta^I(t),v^i(t),\hat{\rm v}^a(t),p_i(t))$ on $M_\mu/G_\mu$, we can recover the complete solution $(x^i(t),\theta^I(t),\theta^A(t),v^i(t),\hat{\rm v}^a(t), p_i(t), \tilde{\rm p}_a(t))$ on $M$ by solving the reconstruction equations
\[
{\dot \theta}^A=\hat{\rm v}^a L^A_a - {\dot x}^i \Lambda^A_i, \quad \tilde{\rm p}_a = \mu_a.
\]
The first  equation can only be given a geometrically concise interpretation if we make use of a principal connection on the principal bundle $TQ\to TQ/G_\mu$, but we will not go into the details of this procedure (we refer the interested reader to~\cite{CrMe2008} for a similar situation in the case of standard Routh reduction).

\section{Special cases}\label{sec:cases}

In this section, we shall obtain the implicit Lagrange-Routh equations in some particular cases and relate the resultant equations, when possible, to the results derived elsewhere in the literature on Routh reduction.

\paragraph{The regular case.}
Let us consider the case where the Lagrangian is regular with respect to the group variables. More precisely, a Lagrangian is said to be {\em $G$-regular} if the Hessian $(\tilde{E}^{\rm V}_a(\tilde{E}_b^{\rm V}(L)))$ is non-singular everywhere on $TQ$, which in coordinates can be expressed as
\[
\det \left(\frac{\partial ^2 L}{\partial \theta^a\partial\theta^b}\right)\neq 0.
\]
 $G$-regularity is probably the weakest condition that the Lagrangian should satisfy to have an analogy with the classical procedure of Routh. In this case, one of the implicit Euler-Lagrange equations (\ref{eq:invLR}), namely the relation  $\hat{E}_b^{\rm V}(R^\mu) =0$ (or, equivalently, $\tilde{E}_a^{\rm V}(L) = \mu_a$) can locally be rewritten in either one of the following explicit forms  $\tilde{\rm v}^a = \tilde{\iota}^a_\mu(q,v^i)$ or  $\hat{\rm v}^a = \hat{\iota}^a_\mu(q,v^i)$, where $\tilde{\iota}^a_\mu, \hat{\iota}^a_\mu$ are smooth functions of $(q,v^i)$. This defines a submanifold $N_\mu$ of $TQ$, with inclusion $\iota_\mu$. We can introduce the new function ${\bar R}^\mu = R^\mu\circ \iota_\mu$ on $N_\mu$. This is the function which is commonly called the {\bfi Routhian} (see, e.g.,\ the papers \cite{CrMe2008,LaCaVa2010,MaRaSc2000}). It is easy to see that, for its reduced version on $N_\mu/G_\mu$, we obtain
\[
\frac{\partial \bar {R}^\mu}{\partial x^i} = \left(\frac{\partial {R}^\mu}{\partial x^i} \circ\iota_\mu\right) + \left( \frac{\partial {R}^\mu}{\partial \hat{\rm v}^a} \circ\iota_\mu\right) \frac{\partial {\hat\iota}^a_\mu}{\partial x^i}.
\]

In view of the fact that $\partial{{ R}^\mu}/\partial{\hat{\rm v}^a} =0$ is part of the reduced implicit Lagrange-Routh equations (\ref{LRred}), every instance of $\partial{{R}^\mu}/\partial{x^i}$ in (\ref{LRred})
can be replaced by $\partial{{\bar R}^\mu}/\partial{x^i}$, and similarly for $\partial{{\bar R}^\mu}/\partial{v^i}$.
The remaining reduced equations are then simply given by
\begin{equation} \label{regred}
{\dot x}^i=v^i, \quad {\dot \theta}^I={\tilde\iota}^J_\mu(q,v^i) K^I_J - {\dot x}^i \Lambda^I_i, \quad
\frac{d}{dt}\left(\frac{\partial {\bar R}^\mu}{\partial v^i}\right)=\frac{\partial {\bar R}^\mu}{\partial x^i} - \Lambda^I_i \frac{\partial {\bar R}^\mu}{\partial \theta^I} - \mu_a B_{ij}^a {\dot x}^j.
\end{equation}
The above equations can be found in~\cite{CrMe2008}.

An even stronger regularity condition is the one used in \cite{LaCaVa2010}. Define for each $v_q\in T_qQ$ a map $\mathcal{J}_L^{v_q}:\mathfrak{g}\to \mathfrak{g}^*$ as follows:
\begin{align*}
\mathcal{J}_L^{v_q}:\mathfrak{g}\to \mathfrak{g}^*\,;\;
\xi\mapsto J_L\left(v_q + \tilde\xi(q)\right).
\end{align*}
In the above, $J_L: TQ \to \mathfrak{g}^{\ast}$ is a standard momentum map on $TQ$ defined by $J_L=J \circ \mathbb{F} L$, where $J: T^{\ast}Q \to \mathfrak{g}^{\ast}$ is the standard momentum map. If we assume that $\mathcal{J}_L^{v_q}$ is a diffeomorphism for every $v_q\in TQ$, it has been shown in \cite{LaCaVa2010} that it is possible to realize the previous equations as the symplectic reduction of the original Lagrangian system. Simple mechanical systems, for which the Lagrangian is of the form $L=T-V$ with $T$ given by a Riemannian metric on $Q$, satisfy automatically this stronger form of $G$-regularity (this follows easily from positive-definiteness of the metric). For a detailed study of the Routh reduction for simple mechanical systems, see~\cite{MaRaSc2000}.

\paragraph{The Abelian case.}
When the group of symmetries is Abelian we have $G_\mu=G$. In this case, there are no coordinates $\theta^I$ and we can just write $\theta^a$ everywhere. The equations for the curve $(x^i(t),v^i(t),\hat{\rm v}^a(t), p_i(t))$ become
\[
{\dot x}^i=v^i, \quad {\dot p}_i=\frac{\partial R^\mu}{\partial x^i}  - \mu_a B_{ij}^a {\dot x}^j, \quad p_i = \frac{\partial R^\mu}{\partial v^i}, \quad \frac{\partial R^\mu}{\partial \hat{\rm v }^a} =0.
\]
Note that even when the group is Abelian, there remains a curvature term in the equations. From the reconstruction equations
\[
{\dot \theta}^a=\hat{\rm v}^b L^a_b - {\dot x}^i \Lambda^a_i, \quad \tilde{\rm p}_a = \mu_a,
\]
we can determine $\theta^a(t)$ and $\tilde{\rm p}_a(t)$. In the case where $L$ is $G$-regular, the manifold $N_\mu$ can be identified with $T(Q/G)$ and the reduced equations (\ref{regred}) can be regarded as the Euler-Lagrange equations of the Routhian ${\bar R}^\mu$ on $T(Q/G)$ subjected to a {\it gyroscopic term} arising from the curvature of $\Lambda$.

An important case of an Abelian symmetry group is that when the Lagrangian has {\it cyclic coordinates}. In this case we have a configuration manifold that is a product $Q=S \times G$, and the Lagrangian is assumed to be invariant under the action of $G$ on the second factor. The equations for $(x^i(t),v^i(t),\hat{\rm v}^a(t), p_i(t))$ become then
\[
{\dot x}^i=v^i, \quad {\dot p}_i=\frac{\partial R^\mu}{\partial x^i} , \quad p_i = \frac{\partial R^\mu}{\partial v^i}, \quad \frac{\partial R^\mu}{\partial \hat{\rm v }^a} =0.
\]
Again, if we have $G$-regularity, these might be further simplified to yield the Euler-Lagrange equations of ${\bar R}^\mu$ (with no gyroscopic term).
\bigskip

\section{Routh-Dirac reduction}\label{sec:Dirac}

In this section, we shall show how the reduced implicit Lagrange-Routh equations can be obtained as a reduced Lagrange-Dirac dynamical system. We will call this type of reduction {\bfi Routh-Dirac reduction}. The procedure we follow relies on some results known for implicit Hamiltonian systems. Again, we assume the same setting as before: a free and proper action of a Lie group $G$ on $Q$, which leaves the Lagrangian $L$ invariant.

\paragraph{Dirac dynamical system.}
We refer to \cite{Co1990,Do1993} for the original works on the notion of a Dirac structure and to \cite{New2,YoMa2006a,CeRaYo2015} for more details on Dirac structures and the dynamics associated to them.
 A linear Dirac structure on a vector space $V$ is a subspace $D_V$ of $V\oplus V^*$ which is a Lagrangian subspace with respect to the pairing
\[
 \langle\langle (v_1,\alpha_1), (v_2,\alpha_2) \rangle\rangle \,= \langle \alpha_2,v_1 \rangle + \langle \alpha_1,v_2 \rangle.
\]
A Dirac structure on a manifold $M$ is  a subbundle $D_M\subset TM\oplus T^*M$ such that $D_M(m)\subset T_mM\times T_m^*M$ is a linear Dirac structure on $T_mM$ at each $m\in M$. In what follows, we will use the terminology {\it Dirac dynamical system}, as in \cite{CeRaYo2015}, to refer to a wide class of implicit Lagrangian or Hamiltonian systems that can be defined in the context of Dirac structures. Given an energy form  $\varphi\in \Gamma(T^*M)$, the dynamics of the {\bfi Dirac dynamical system} $(\varphi, D_{M})$ is given by the following condition on a curve $c:I\subset \mathbb{R}\to M$:
\[
\dot c(t)\oplus  \varphi\left(c(t)\right)\in D_M(c(t)).
\]

\paragraph{Implicit Euler-Lagrange equations.}
Using the above, we can obtain the implicit Euler-Lagrange equations, by taking $M$ to be the Pontryagin bundle $M=TQ\oplus T^*Q$ over $Q$. In order to construct $D_{M}$ we consider  the pullback $\Omega_M =\pi_{T^{\ast}Q}^*\Omega_{T^{\ast}Q}$ of the canonical symplectic form $\Omega_{T^{\ast}Q} =-\mathbf{d}\theta_{T^{\ast}Q} $ on $T^*Q$ to $M$ by using the canonical projection $\pi_{T^{\ast}Q}: M \to T^*Q$. Then, $\Omega_M$  is a presymplectic form on $M$, which naturally defines a Dirac structure $D_M$ on $M$ by means of the graph of $\Omega_{M}$: for each $(q,v,p) \in M$, set
\begin{equation*}
D_{M}(q,v,p)
=
\{ ((\delta{q}, \delta{v}, \delta{p}), (\alpha, \beta, \gamma)) \in TM \oplus T^{\ast} M \mid
 \alpha + \delta{p} = 0,\,\gamma = \delta{q},\, \beta = 0 \}.
\end{equation*}
The {\it generalized energy function} $\mathcal{E}_L(q,v,p) =\langle p,v\rangle - L(q,v)$ on $M$ defines the energy form  $\mathbf{d}\mathcal{E}_L$. When we express the two-form $\Omega_{T^{\ast}Q}$ in natural coordinates as ${d}q^\alpha \wedge {d} p_\alpha$ it is  easy to check that a curve $c(t)=(q(t),v(t),p(t))$ in $M$ is a solution of the implicit Euler-Lagrange equations (\ref{eq:IL1}) if $(\mathbf{d}\mathcal{E}_L, D_{M})$ satisfies the so-called Lagrange-Dirac dynamical system
\begin{equation}\label{eq:DIL}
{\dot c}(t) \oplus \mathbf{d}\mathcal{E}_L(c(t)) \in D_M(c(t)),\quad \text{for all $t$}.
\end{equation}
On the other hand, when we write the generalized energy in quasi-velocities and quasi-momenta as  $\mathcal{E}_{L}={\rm p}_\alpha {\rm v}^\alpha - L$  and $\mathbf{d}\mathcal{E}_L={\rm v}^\alpha d{\rm p}_\alpha+{\rm p}_\alpha d{\rm v}^\alpha-\mathbf{d}L$, and likewise the two-form $\Omega_{T^{\ast}Q}$  as
\begin{equation}
\Omega_{T^{\ast}Q} =   W^\alpha \wedge d{\rm p}_\alpha   + \frac{1}{2}R^\alpha_{\beta\gamma} {\rm p}_\alpha W^\beta \wedge W^\gamma\label{eq:omega}
\end{equation}
we may easily obtain from (\ref{eq:DIL}) the implicit Euler-Lagrange equations with quasi-velocities and quasi-momenta (\ref{eq:IL}). Note that $W^\alpha$ has been introduced before as a 1-form on $Q$, while in the above we think of $W^\alpha$ as a (semi-basic) 1-form on $T^*Q$. We will use this slight abuse of notation from now on.

\paragraph{Reduction of Dirac structures.} In this paragraph we assume that $M$ is an arbitrary manifold, endowed with a Dirac structure $D_M$. We first recall some generalities on the reduced Dirac structure induced by a Lie group action. The results here were originally developed in \cite{Va1998,BlVa2001}.
\medskip

Assume that a Lie group $G$ acts (freely and properly) on the manifold $M$, and denote by $\phi_g(m)=g\cdot m$ the action of $g\in G$ on a point $m\in M$. This action lifts naturally by tangent and cotangent lifts to $TM\oplus T^*M$. Assume also that  the action admits an equivariant momentum map relative to $D_{M}$, that is, assume that there exists a $G$-equivariant map $J^M : M \rightarrow \mathfrak{g}^\ast $ such that $ \tilde{\xi} \oplus \mathbf{d}J^M_\xi  \in D_{M}$, for all $\xi \in \mathfrak{g}  $, where we recall that $\tilde\xi$ stands for the infinitesimal generator of the action associated with $\xi \in \mathfrak{g}$ and  $J^M_{\xi}$ is the smooth function on $M$ defined by $J^M_{\xi}(m)=\left<J^M(m), \xi \right>,\;m \in M$.

 We assume now that $D$ is invariant under $G$. The Dirac reduction procedure will be carried out in the following two steps. First, if $ \mu \in \mathfrak{g}^\ast $ is a regular value of $J^M$, then $M_{\mu}=(J^M)^{-1} (\mu) \subset M$ is a submanifold. If the vector subspace $D_{M}(m) \cap (T_m {M_{\mu}} \times T^*_m M|_{M_{\mu}})\subset T_m M_{\mu} \times T^*_m M|_{M_{\mu}}$ has constant dimension at each $m\in M_{\mu}$, then these vector spaces naturally induces a {\bfi restriction of the Dirac structure} $D_{M_{\mu}} \subset TM_{\mu} \oplus T^*M_{\mu}$. Second, one observes that the Dirac structure $D_{M_{\mu}}$ is $G _\mu $-invariant since
$$
D_{M_{\mu}}({g\cdot m})=g\cdot D_{M_{\mu}}(m),
$$
where $G_ \mu =\{g \in G \mid \operatorname{Ad}^*_ g \mu = \mu \}$ is the coadjoint isotropy subgroup  of $ \mu $. This leads to a {\bfi reduced Dirac structure} $D_{M_{\mu}/G _\mu } \subset T(M_{\mu}/G _\mu ) \oplus  T^\ast (M_{\mu}/G _\mu )$ on the reduced space $M_{\mu}/G _\mu = J^{-1} (\mu) /G _\mu $, which is given by
\begin{equation}\label{eq:Rstruc}
\begin{aligned}
D_{M_{\mu}/G _\mu }:&=\{(X, \alpha ) \in \mathfrak{X}(M_{\mu}/G _\mu ) \times \Omega^1(M_{\mu}/G_ \mu )\mid \exists \,(Y, \beta  )\in D_{M}, \\
& \qquad \qquad \qquad \qquad \qquad \qquad \;\text{such that}\; T \pi_\mu \circ Y = X \circ \pi_\mu , \; \pi_\mu ^\ast \alpha = \beta  \},
\end{aligned}
\end{equation}
where $\pi_\mu:M_\mu\to M_\mu/G_\mu$ is the canonical projection, which is a surjective submersion.

\paragraph{Symmetry reduction of implicit Hamiltonian systems.}
An important case of Dirac dynamical systems is that when the energy section is given by the differential of a Hamiltonian function $H$ on $M$. Let $H$ be a $G$-invariant Hamiltonian on $M$ and let  $c(t)$ be  a solution curve for the {\bfi implicit Hamiltonian system} $(\mathbf{d}H, D_{M})$, i.e.\ a curve that satisfies
$$
\dot{c}(t) \oplus \mathbf{d}H(c(t)) \in D_{M}(c(t)).
$$
Then,
$$
\frac{dJ^M_{\xi}}{dt}(c(t))=\left< \mathbf{d}J^M_{\xi}(c(t)), \dot{c}(t)\right>=-\left< \mathbf{d}H, \tilde{\xi}\right>(c(t))=0, \;\; \textrm{for all $t$ and $\xi \in \mathfrak{g}$.}
$$
So, $J^M_{\xi}$ is a {\it first integral} of the implicit Hamiltonian system and we can restrict the Dirac dynamical system $(\mathbf{d}H, D_{M})$ to $(\mathbf{d} H_{\mu}, D_{M_{\mu}})$, where $H_{\mu}=H|_{M_{\mu}}$. Next, we can reduce the restricted implicit Hamiltonian system $(\mathbf{d}H_{\mu}, D_{M_{\mu}})$ to obtain the {\bfi reduced implicit Hamiltonian system} $(\mathbf{d}\mathfrak{H}_{\mu}, D_{M_{\mu}/G_{\mu}})$, which satisfies
\begin{equation} \label{eq:redHam}
\dot{\mathfrak{c}}(t) \oplus \mathbf{d}\mathfrak{H}_{\mu}(\mathfrak{c}(t))) \in D_{M_{\mu}/G_{\mu}}(\mathfrak{c}(t),
\end{equation}
for each $\mathfrak{c}(t)=\pi_{\mu}(c(t))$ in $M_{\mu}/G_{\mu}$, where $\mathfrak{H}_{\mu} \circ \pi=H_{\mu}$ is the reduced Hamiltonian on $M_{\mu}/G_{\mu}$.

More details on this reduction of Dirac structures and its associated reduced implicit Hamiltonian systems can also be found in \cite{BlRa2004}.

\paragraph{The reduced implicit Lagrange-Routh equations.}
We consider again the Dirac structure $D_M$ on the Pontryagin bundle $M=TQ\oplus T^*Q$ given by the graph of the presymplectic form $\Omega_M=\pi_{T^{\ast}Q}^*\Omega_{T^{\ast}Q}$. For $J^M: M \to \mathfrak{g}^\ast$, we take $J^M=J\circ\pi_{T^\ast Q}$, where $J: T^\ast Q \to \mathfrak{g}^\ast$ stands for the standard momentum map on $Q$. The goal of this paragraph is to demonstrate that, in this particular setting, the reduced implicit Hamiltonian system (\ref{eq:redHam}), with Hamiltonian $H={\mathcal E}_L$, is nothing but the system given by the reduced implicit Lagrange-Routh equations (\ref{LRred}).

It is well-known that if $D_{M}$ is a Dirac structure given by the graph of a symplectic form $\Omega_{M}$ the reduced Dirac structure $D_{M_{\mu}/G _\mu }$ may be given by the graph of a reduced symplectic form $\Omega_{M_{\mu}/G _\mu }$. The following observations may be obtained:
\begin{itemize}
\item[i)]
The action of $G$ on $M=TQ\oplus T^*Q$ restricts to a $G_\mu$-action on $M_\mu$ by tangent and cotangent lifts, and this action leaves the presymplectic form $\Omega_{M_\mu}$ invariant. Indeed, since $\Omega_{M_\mu}=\imath^*\Omega_M$ (where $\imath:M_\mu\to M$ is the inclusion), its invariance follows directly from the $G_\mu$-invariance of $\Omega_{M}$.

\item[ii)] Moreover, one can show that the form $\Omega_{M_\mu}$ drops to $M_\mu/G_\mu$. Indeed, it suffices to check that $\Omega_{M_\mu}$ annihilates vectors which are vertical to the fibration $\pi_\mu:M_\mu\to M_\mu/G_\mu$. Thus the {\it reduced presymplectic form} $\Omega_{{M_\mu}/{G_\mu}}$ is defined on $M_\mu/G_\mu$ from the invariance of $\Omega_{M_\mu}$.

\item[iii)]
It follows immediately that $D_{M_\mu}$ is $G_\mu$-invariant. In particular, we can define a {\bfi reduced Dirac structure} $D_{{M_\mu}/{G_\mu}}$ on $M_\mu/G_\mu$,
    \[
   D_{{M_\mu}/{G_\mu}}\subset T(M_\mu/G_\mu)\oplus T^*(M_\mu/G_\mu)
    \]
such that
\[
\pi^{\ast}_{\mu}D_{M_{\mu}/G_{\mu}}=\iota^{\ast}D_{M}.
\]
Note that the above characterization of $D_{M_\mu/G_\mu}$ can also be obtained in the standard cases of symplectic and presymplectic reduction by the well-known characterization of the reduced (pre)symplectic form, to be found in \cite{MaWe1974} and \cite{EcMuRo1999}). We plan to investigate whether the same result may also hold for almost Dirac structures with regular distributions.
\end{itemize}

To write down the reduced system of the implicit Hamilton system~\eqref{eq:DIL}, we first need to give a coordinate expression of the two 2-forms $\Omega_{M_\mu}$ and $\Omega_{M_\mu/G_\mu}$ whose graph define the Dirac structures $D_{M_\mu}$ and $D_{M_\mu/G_\mu}$.

Using a principal connection $A$ on the bundle $\pi:Q\to Q/G$, one identifies $J^{-1}(\mu)\simeq T^*(Q/G)\times_{Q/G}Q$. Under this identification the presymplectic form reads
\[
\Omega_{M_\mu}=\Omega_{Q/G} - \mathbf{d}A_\mu,
\]
where we have used a slight abuse of notations and omitted the pullbacks from the spaces $T^*(Q/G)$ and $Q$ where the forms $\Omega_{Q/G} $ and $\mathbf{d}A_\mu$ are defined, respectively. From Cartan's structure equation, it follows $\mathbf{d}A_\mu=\langle \mu,[A,A]\rangle-B_\mu$. Therefore, in coordinates, we have
\begin{equation}\label{eq:omega_N}
\Omega_{M_\mu}  =   {d}x^i \wedge {d}p_i + \frac{1}{2}\mu_a \left( B^a_{ij} {d}x^i \wedge {d}x^j -  C^a_{bc} {\tilde E}^b \wedge {\tilde E}^c \right),
\end{equation}
where $\{dx^i, {\tilde E}^a\}$ stand for the dual of the basis $\{X_i, {\tilde E}_a\}$ and where we have identified $X^i=dx^i$ in the notations of the previous sections. We point out that, again, we do not write explicitly the pullbacks: one should think of the forms ${\tilde E}^a$ in~\eqref{eq:omega_N} as semi-basic forms on $T^*Q$ or, equivalently, as the vertical lifts $(\tilde E^a)^{\rm V}$ to $T^*Q$ of the corresponding forms in $Q$ (see~\cite{YaIs1973} for more details). To ease the notation we will keep this convention from now on since the coordinate expressions agree, unless there is risk of confusion. The expression~\eqref{eq:omega_N} is a particular instance of the expression~\eqref{eq:omega} in the current frame.

Let $R^{\mu}$ be the generalized Routhian given in \eqref{Routhian} and consider the energy section determined by the differential of the energy $\mathcal{E}_L$ restricted to $M_\mu$. Writing $\mathcal{E}_\mu=\imath_\mu^*\mathcal{E}_L$ in terms of the Routhian as $\mathcal{E}_\mu=p_iv^{i}-R^\mu$, and using a formula similar to the one we had for calculating variations of $L$ in terms of the anholonomic frame (see~\eqref{eq:varL}), it follows
\begin{align}\label{eq:d_energy}
\mathbf{d}\mathcal{E}_\mu=&-\left(X_i^{\rm C}(R^\mu)+ \tilde{E}_a^{\rm V}(R^\mu)B^a_{ij}v^j\right)dx^i - \left(X_i^{\rm V}(R^\mu)-p_i\right)dv^i \nonumber \\
&\hspace{2cm}+\left(\mu_a+\tilde{E}_a^{\rm V}(R^\mu)\right)C^a_{bc}\tilde{\rm v}^c\tilde E^b - \tilde{E}_a^{\rm V}(R^\mu)d \tilde{\rm v}^a + v^idp_i,
\end{align}
where we have used the relation $\tilde{E}^{\rm C}_a(R^\mu)=   - \mu_c C^c_{ab}\tilde{\rm v}^b$ from~\eqref{eq:Rder}.

\medskip

Recall that $D_{M_\mu}\subset TM_\mu\oplus T^*M_\mu$ is the Dirac structure induced from $\Omega_{M_\mu}$. Before we continue with the expression for $\Omega_{M_\mu/G_\mu}$, it is instructive to have a look at the first step in the reduction, namely the restriction of the dynamics to $M_\mu$. Consider for that reason the Lagrange-Dirac dynamical system $(\mathbf{d}\mathcal{E}_{\mu}, D_{M_\mu})$, which satisfies, for each $c(t) = (x^i(t), \theta^a(t), v^i(t), \tilde{\rm v}^a(t), p_i(t) )$ in $M_\mu$,
$$
\dot c(t) \oplus \mathbf{d}\mathcal{E}_\mu\left(c(t)\right)\in \left(D_{M_\mu}\right)_{\left(c(t)\right)}.
$$
It leads to the following set of equations:
\begin{eqnarray*}
&& \dot x^i=v^i,\qquad  \tilde{E}_a^{\rm V}(R^\mu)=0,\qquad  X_i^{\rm V}(R^\mu)-p_i=0,\\
&& \mu_aB^a_{ij}\dot x^j + \dot p_i = X_i^{\rm C}(R^\mu)+\tilde{E}_a^{\rm V}(R^\mu)B^a_{ij}v^j, \qquad  \left(\mu_a+\tilde{E}_a^{\rm V}(R^\mu)\right)C^a_{bc}\tilde{\rm v}^c =\mu_aC^a_{bc}\tilde{\rm u}^c,
\end{eqnarray*}
where $\tilde{\rm u}^b = (K^{-1})^b_a (\dot{\theta}^a+  \dot{x}^i \Lambda^a_i)$. The above equations might be further simplified as
 \begin{equation}\label{eq:RD}
 \begin{split}
 &\mu_aC^a_{bc}\tilde{\rm v}^c=\mu_aC^a_{bc}\tilde{\rm u}^c, \qquad  \tilde{E}_a^{\rm V}(R^\mu) = 0,  \\
 &v^i=\dot{x}^i, \qquad  X_i^{\rm V}(R^\mu)=p_i, \qquad  \dot{p}_i= X_i^{\rm C}(R^\mu) - \mu_a B^a_{ij}v^j.
 \end{split}
 \end{equation}

The similarity with the implicit Lagrange-Routh equations~\eqref{eq:LR} is obvious, with the only difference that it is not possible to conclude from equations~\eqref{eq:RD} that $\tilde{\rm v}^c=\tilde{\rm u}^c$. This is a consequence of the fact that solutions of a presymplectic equation are only determined up to elements in the kernel of the presymplectic form (see~\cite{GoNe1979}). Indeed, considering the splitting $\mathfrak{g}=\mathfrak{g}_\mu\oplus (\mathfrak{g}/\mathfrak{g}_\mu)$ introduced earlier in \S\ref{sec:redsection}, the relation $\mu_aC^a_{bc}(\tilde{\rm v}^c-\tilde{\rm u}^c)=0$ implies that $\tilde{\rm v}^I=\tilde{\rm u}^I$, but it is not true in general that also $\tilde{\rm v}^A=\tilde{\rm u}^A$. This is reminiscent of the standard case of symplectic reduction on $T^*Q$; if $\imath_\mu:J^{-1}(\mu)\to T^*Q$ denotes the inclusion, then the kernel of the presymplectic form $\Omega_\mu=\imath_\mu^*\Omega_Q$ is given by $\ker \Omega_\mu=\{\tilde E\mid E\in\mathfrak{g}_\mu\}$.
\medskip

The equations (\ref{eq:RD}) only refer to the restriction step in the reduction process of an implicit Hamilton system. As explained before, in the second step, we need to reduce that system by $G_\mu$. This boils down to looking at the Dirac structure defined by the graph of $\Omega_{M_\mu/G_\mu}$.

\begin{proposition}
Let $\mathfrak{E}_\mu$ be the reduced energy function on $M_\mu/G_\mu$ defined by $\mathfrak{E}_\mu \circ \pi_{\mu}=\mathcal{E}_{\mu}$.
Given the Dirac structure $D_{{M_\mu}/{G_\mu}}$, a curve $\mathfrak{c}(t)=\pi_{\mu}(c(t))$ in ${M_\mu}/{G_\mu}$ is a solution curve of the {\bfi Routh-Dirac dynamical system} $(\mathbf{d}\mathfrak{E}_\mu, D_{{M_\mu}/{G_\mu}})$, which  satisfies
$$
\dot {\mathfrak{c}}(t) \oplus \mathbf{d}\mathfrak{E}_\mu\left(\mathfrak{c}(t)\right)\in \left(D_{{M_\mu}/{G_\mu}}\right)\left(\mathfrak{c}(t)\right),
$$
if and only if, the reduced implicit Lagrange-Routh equations~\eqref{LRred} hold.
\end{proposition}

\begin{proof}
Using the splitting $(E_a)=(E_I,E_A)$ of $\mathfrak{g}$, we write the presymplectic form $\Omega_{M_\mu}$ in $M_\mu$ as
\[
\Omega_{M_\mu}  = {d}x^i \wedge {d}p_i + \frac{1}{2} \mu_a \left( B^a_{ij} {d}x^i \wedge {d}x^j -
C^a_{IJ} {\tilde E}^I \wedge {\tilde E}^J \right).
\]
Recall that, with the notations in \S\ref{sec:implicit}, we have $\hat E_a=A_a^b \tilde E_b$. Then we have, for duals, $\tilde E^I=A^I_J \hat E^J$. With this we can rewrite the previous expression of $\Omega_{M_\mu}$ in terms of the forms $\hat E^I$:
\[
\Omega_{M_\mu}  = {d}x^i \wedge {d}p_i + \frac{1}{2} \mu_a \left( B^a_{ij} {d}x^i \wedge {d}x^j -
A^I_K A^J_L C^a_{IJ} {\hat E}^K \wedge {\hat E}^L \right).
\]
In this frame, the reduced two-form $\Omega_{M_\mu/G_\mu}$ has formally the same expression.

The proof will follow from the expression~\eqref{eq:d_energy} for $\mathbf{d} \mathcal{E}_\mu$ in invariant (hat) coordinates, and from its reduction to $ \mathbf{d}\mathfrak{E}_\mu$. This computation is completely analogous to the one carried out before in \S\ref{sec:redsection}. If we take into account that $B^a_{ij} = A^a_b \hat B^b_{ij}$, $\tilde {\rm v}^I=A^I_J \hat{\rm v}^J$, etc., it follows that
\begin{align*}
\mathbf{d}\mathcal{E}_\mu = &-\left(  \frac{\partial R^\mu}{\partial x^i} - \Lambda_i^I \frac{\partial R^\mu}{\partial \theta^I} \right)dx^i - \left( \frac{\partial R^\mu}{\partial v^i}- p_i\right)  dv^i   -\mu_a A^I_K A^J_L C^a_{cd}\hat{\rm v}^L\hat{E}^K  \\ &\hspace{2cm}  - \frac{\partial R^\mu}{\partial \hat{\rm v}^b}  C^a_{cd} (A^{-1})^b_a A^c_K A^d_L \hat{\rm v}^L \hat{E}^K   -  \frac{\partial R^\mu}{\partial \hat{\rm v}^a} d \hat{\rm v}^a + v^i dp_i.
\end{align*}
Using the relations for the structure constants derived in \S\ref{sec:redsection}, it is then straightforward to check that a curve
\[
\mathfrak{c}(t) = (x^i(t), \theta^I(t), v^i(t), \hat{\rm v}^a(t), p_i(t) ):I\subset \mathbb{R}\to {M_\mu}/{G_\mu},
\]
satisfies (\ref{eq:redHam}), if and only if the implicit Lagrange-Routh equations~\eqref{LRred} hold. One of the equations will appear rather as
 $\mu_aC^a_{IJ}(\hat{\rm v}^I-\hat{\rm u}^I)=0$, but, as we mentioned before,  this implies  $\hat{\rm v}^I=\hat{\rm u}^I$.
\end{proof}

\paragraph{Remarks.}
We will call the reduction process that we have just described {\bfi Routh-Dirac reduction}. It is a particular instance of a reduced Dirac dynamical system. One can see a similar discussion on the Dirac reduction associated with the symplectic reduction in the more general context of Dirac anchored vector bundle reduction in \cite{CeRaYo2015}.

It should be remarked that the above Dirac structure $D_{M_{\mu}/G _\mu } $ is different from the one used in \cite{YoMa2007, YoMa2009}, where the Dirac structure is rather defined as a subbundle of the vector bundle $(TM \oplus T^*M)/G$. This last definition may be advantageous when one wants to give a geometric interpretation of the so-called implicit Lagrange-Poincar\'e or Hamilton-Poincar\'e equations in the variational link with Dirac structures. These two sets of equations are the result of a reduction process that takes the full symmetry group $G$ of a mechanical system into account, but it does not produce {\it conserved quantities} $\mu$. Our aim, in the previous sections, was a {\it Routh-type reduction}; so we wanted to take full advantage of the conserved quantities, at the price of reducing by a possibly smaller symmetry group $G_\mu$.

\section{Examples}\label{sec:examples}
To conclude, we will now discuss two illustrative examples where the regularity conditions on the Lagrangian fail.

\paragraph{Linear Lagrangians.}
A linear Lagrangian is probably the easiest case where $L$ is not locally $G$-regular. In general, a linear Lagrangian $L$ on $TQ$ conveys to the form $L=\langle \alpha(q),v_q\rangle - f(q)$ for some 1-form $\alpha$ on $Q$ and function $f$ on $Q$. We say a few words about the two examples we mentioned in the Introduction.

\medskip

The Lagrangian
\[
L(x,y,v_x,v_y)=(v_x)^2+v_x v_y- V(x)
\]
of the first example has a cyclic coordinate $y$. Fix a value $\mu\in\mathbb{R}$ for the momentum $p_y$. The Routhian, constructed with respect to the trivial connection on $\mathbb{R}^2\to\mathbb{R}$, reads
\[
R^\mu(x,y,v_x,v_y)=(v_x)^2+v_x v_y-\mu v_y - V(x),
\]
and the reduced implicit Routh  equations are then
\[
\dot x=v_x,\qquad \dot p_x=\frac{\partial R^\mu}{\partial x}=-V'(x),\qquad  p_x=2v_x+v_y, \qquad v_x-\mu=0.
\]
Together with the reconstructions equations $\dot y=v_y$ and $p_y=\mu$, we get a solution of the original implicit Euler-Lagrange equations for $L$ with the prescribed value of the momentum $p_y$.

\medskip

Consider again the Lagrangian for the dynamics of point vortices in the plane from \cite{Ch1978}:
\[
L(z_l,{\dot z}_l)=\frac{1}{2i}\sum_k\gamma_k\big(\bar z_k\dot z_k-z_k\overline{\dot z_k}\big)-\frac{1}{2}\sum_n\sum_{k\neq n}\gamma_n\gamma_k\text{ln}|z_n-z_k|.
\]
Although it is $S^1$ invariant under rotations on $\mathbb{C}$, the coordinates are not adapted. A possible way to proceed is to take polar coordinates for each position $z_k(t)\in \mathbb{C}$, say $z_k=\rho_k e^{i\theta_k}$, and then to consider the relative angle with respect to $\theta_1$. Defining $\phi_k=\theta_k-\theta_1$ for $k\geq 2$ and $\phi_1=\theta_1$, we have an $S^1$-action along $\phi_1$ with associated momentum $\partial \bar L/\partial\dot \phi_1$, where $\bar L$ is the Lagrangian in the new coordinates. A computation then shows that this conserved quantity is precisely the moment of circulation (also called angular impulse in~\cite{Ne2001})
\[
I=\sum_k\gamma_k \rho_k^2=\sum_k\gamma_k|z_k|^2.
\]
From here the implicit Lagrange-Routh equations follow without further difficulty.

\paragraph{A degenerate Lagrangian.}
We will now discuss a mechanical model for field theories to be found in \cite{CaKo1982, CaKo1987}, although we will follow the exposition in \cite{EcMuRo1999}, where it appears in the context of presymplectic reduction. The Lagrangian is
\[
L=\big(\dot{\bar\psi}\big)^i m_{ij}\big(\dot{\psi}\big)^j + \big(\dot{\bar\psi}\big)^i c_{ij}\big({\psi}\big)^j
- \big({\bar\psi}\big)^i \tilde c_{ij}\big({\dot\psi}\big)^j - \big({\bar\psi}\big)^i \tilde r_{ij}\big({\psi}\big)^j,
\]
where $\psi^i,\bar\psi^i$ represent the scalar complex fields (which are regarded as coordinates in the model), and where the matrices $m_{ij},c_{ij},\tilde c_{ij},r_{ij}$ satisfy: $m_{ij},r_{ij}$ are hermitian and $\overline{(\tilde c_{ij})}=-c_{ji}$. This guarantees that $L$ is real.

We will consider the same values for these matrices that appear in~\cite{EcMuRo1999}:{
\begin{equation*}
  m_{ij} =
  \begin{pmatrix}
  0 & 0 & 0 \\
  0 & m_2  & 0 \\
  0 & 0 &  m_3
 \end{pmatrix},
  \qquad c_{ij} = \tilde c_{ij} = \frac{1}{2}
  \begin{pmatrix}
  0 & 0 & 0 \\
  0 & i  & 0 \\
  0 & 0 &  i
 \end{pmatrix},
   \qquad r_{ij}=
  \begin{pmatrix}
  1 & 0 & 0 \\
  0 & 1  & 0 \\
  0 & 0 &  1
 \end{pmatrix},
\end{equation*}
and which lead to a degenerate Lagrangian. With that choice, it is apparent that $L$ becomes invariant under the actions of $S^1$ by rotation in the second and third scalar fields. Thus, we have a $\mathbb{T}^2$-action on $Q$. Writing $\psi^k=x^k+iy^k$ and $\dot\psi^k=u^k+iv^k$ the Lagrangian becomes:
\begin{align*}
L =\; &m_2\left((u^2)^2+(v^2)^2\right)+m_3\left((u^3)^2+(v^3)^2\right)+v^2x^2+v^3x^3-u^2 y^2 - u^3 y^3\\
&-\left((x^1)^2+(y^1)^2\right)-\left((x^2)^2+(y^2)^2\right)-\left((x^3)^2+(y^3)^2\right).
\end{align*}
In this case, coordinates adapted to the $\mathbb{T}^2$-action are simply the usual polar coordinates on both $(x^2,y^2)$ and $(x^3,y^3)$. We will denote these sets of polar coordinates by $(r,\theta)$ and $(\rho,\phi)$ respectively. Then, it follows
\begin{align*}
L =\; &m_2\left(r^2v_\theta^2+v_r^2\right)+m_3\left(\rho^2v_\theta^2+v_{\rho}^2\right) +r^2v_\theta+\rho^2v_\phi-r^2-\rho^2-\left((x^1)^2+(y^1)^2\right),
\end{align*}
and the equations for $p_\theta$ and $p_\phi$ are of the form
\[
p_\theta=2m_2r^2v_\theta+r^2,\qquad p_\phi=2m_3\rho^2v_\theta+\rho^2.
\]
\medskip

The Routhian reads, for a given choice $(p_\theta,p_\phi)=(\mu_\theta,\mu_\phi)$, and with the trivial connection, as follows:
\[
R^\mu = \left(r^2v_\theta^2+v_r^2\right)+m_3\left(\rho^2v_\theta^2+v_{\rho}^2\right)+
r^2v_\theta+\rho^2v_\phi-r^2-\rho^2-\left((x^1)^2+(y^1)^2\right)-\mu_\theta v_\theta-\mu_\phi v_\phi.
\]
The reduced implicit Lagrange-Routh equations in this case are
\begin{alignat*}{3}
 \dot x^1&=u^1,\qquad\quad & \dot p_{x^1}&=-2x^1,\qquad & p_{x^1}&=0 ,\\
 \dot y^1&=v^1,\qquad & \dot p_{y^1}&=-2y^1,\qquad & p_{y^1}&=0 ,\\
 \dot r&=v_r,\qquad & \dot p_{r}&=2rv_{\theta}^2+2rv_{\theta},\qquad & p_{r}&=2v_r ,\\
 \dot \rho&=v_\rho,\qquad & \dot p_{\rho}&=2rv_{\phi}^2+2\rho v_{\phi},\qquad & p_{\rho}&=2v_\rho ,\\
 2r^2v_\theta+r^2&=\mu_\theta,\qquad & 2\rho^2v_\phi+\rho^2&=\mu_\phi,
\end{alignat*}
and the reconstruction equations are given by
\[
\dot \theta =v_\theta, \qquad \dot \phi =v_\phi, \qquad  p_\theta =\mu_\theta, \qquad  p_\phi =\mu_\phi.
\]
\paragraph{Acknowledgments} We thank Santiago Capriotti for pointing out an inaccuracy in an earlier version of our text. EGTA wants to thank the Czech Science Foundation for funding under research grant No 14-02476S `Variations, Geometry and Physics'. EGTA and TM both acknowledge support from FWO--Vlaanderen. HY is partially supported by JSPS (Grant-in-Aid 26400408), JST (CREST), Waseda University (SR 2014B-162, SR 2015B-183) and MEXT's "Top Global University Project''. This work is part of the IRSES project ``Geomech'' (246981) within the 7th European Community Framework Programme. EGTA and TM are grateful to the Department of Applied Mechanics and Aerospace Engineering of Waseda University for its hospitality during the visits which made this work possible.


\end{document}